\documentclass[12pt]{amsart}
\usepackage{epsfig,amsmath,amsfonts,latexsym}
\usepackage[usenames,dvipsnames]{color}
\usepackage{palatino}
\usepackage[hypcap=false]{caption}

\usepackage{color}

\usepackage[hyperfootnotes=false]{hyperref}
\hypersetup{
  colorlinks,
  citecolor=Purple,
  linkcolor=Black,
  urlcolor=Green}
  
\theoremstyle{plain}
\newtheorem{dummy}{anything}[section]
\newtheorem{theorem}[dummy]{Theorem}

\newtheorem{lemma}[dummy]{Lemma}

\newtheorem{proposition}[dummy]{Proposition}

\newtheorem{corollary}[dummy]{Corollary}

\theoremstyle{definition}
\newtheorem{definition}[dummy]{Definition}

\newtheorem{remark}[dummy]{Remark}

\theoremstyle{remark}

\newcommand{\Z}{\mathbb{Z}}

\newcommand{\M}{\mathcal{M}}
\newcommand{\A}{\mathcal{A}}

\def\a{\alpha}
\def\b{\beta}
\def\d{\delta}

\def\t{\mu}

\begin{document}

\title {On open books for nonorientable $3$-manifolds}

\author{Burak Ozbagci}

\address{Department of Mathematics, Ko\c{c} University, Istanbul,
Turkey}
\email{bozbagci@ku.edu.tr}

\subjclass[2000]{}
\thanks{}


\begin{abstract}
We show that the monodromy of Klassen's genus two open book for $P^2 \times S^1$ is the $Y$-homeomorphism of Lickorish, which is also known as the crosscap slide. Similarly, we  show that $S^2 \widetilde{\times} S^1$ admits a genus two open book whose monodromy is the crosscap transposition.  Moreover, we show that each of  $P^2 \times S^1$  and  $S^2 \widetilde{\times} S^1$ admits infinitely many isomorphic genus two open books whose monodromies are mutually nonisotopic.  Furthermore, we include a simple observation about the stable equivalence classes of open books for  $P^2 \times S^1$  and  $S^2 \widetilde{\times} S^1$. Finally, we formulate a version of Stallings' theorem about the Murasugi sum of open books,  without imposing any orientability assumption on the pages. \end{abstract} 


\maketitle

\section{Introduction}\label{sec: intro}

It is a classical theorem of Alexander \cite{a} that there is an open
book for any closed orientable $3$-manifold, which can be obtained by pulling back the standard open book for $S^3$ via some branched
covering, where the branch set is braided about the binding.  In the nonorientable case,  an analogous result was obtained by Berstein and Edmonds \cite{be}, who first proved that every closed nonorientable $3$-manifold is a branched cover of $P^2 \times S^1$ and then showed as a corollary that an open book for $P^2 \times S^1$ can be pulled back to the cover,  where $P^2$ denotes the real projective plane.  In their work,  however, Berstein and Edmonds used the Stallings' fibration theorem \cite{s} to show the existence of a genus two open book with connected binding (meaning that its page is  a Klein bottle with one hole)  for $P^2 \times S^1$,  rather than describing the open book explicitly. They also mentioned \cite[Remark 9.2]{be} that the monodromy of their genus two open book must be the $Y$-homeomorphism of Lickorish \cite{l}, without a detailed proof.



In a two-page note, Klassen  \cite{k} described an explicit genus two open book with connected binding for $P^2 \times S^1$,   {\em without discussing its monodromy}. In this paper, we prove that the monodromy of Klassen's genus two open book for $P^2 \times S^1$ is the $Y$-homeomorphism---which is a primary example of a surface homeomorphism that cannot be factorized into Dehn twists.  We also show that the nonorientable $S^2$-bundle over $S^1$, denoted by $S^2 \widetilde{\times} S^1$, admits a genus two open book with connected binding, whose monodromy is the crosscap transposition---another example of a surface homeomorphism that cannot be factorized into Dehn twists. As a matter of fact,  we show that each of  $P^2 \times S^1$  and  $S^2 \widetilde{\times} S^1$ admits infinitely many isomorphic genus two open books with connected bindings, whose monodromies are mutually nonisotopic.

The reader should compare our result with a recent result of Ghanwat, Pandit and Selvakumar \cite{gps},  who proved that every closed nonorientable $3$-manifold admits a genus one open book (meaning that its page is a projective plane with holes), which is analogous to the fact that every closed orientable $3$-manifold admits a planar open book (cf. \cite{r}, see also \cite{o}). A notable common feature of the planar open books constructed in \cite{o} for the orientable $3$-manifolds and the genus one open books constructed in  \cite{gps} for the nonorientable $3$-manifolds is that the monodromy of each one of them is a product of Dehn twists along two-sided curves on the page. 

Giroux showed that on a closed orientable $3$-manifold,  the coorientable contact structures up to isotopy are in one-to-one correspondence with the open books up to positive stabilization. In contrast,  there is no contact structure on a {\em nonorientable} $3$-manifold even if one drops the usual coorientability assumption on the contact structures.   Nevertheless, one can still consider an equivalence relation on the set all open books on a closed nonorientable $3$-manifold, induced by stabilizations. According to \cite{gps},  $P^2 \times S^1$ admits a genus one  open book whose monodromy is a product of Dehn twists along two-sided curves. On the other hand, as we show here, the monodromy of Klassen's genus two open book  for  $P^2 \times S^1$ is  the $Y$-homeomorphism. As a consequence, we conclude that  these two open books for  $P^2 \times S^1$ cannot be in the same equivalence class.  By a similar argument we show that the genus two open book with connected binding for $S^2 \widetilde{\times} S^1$ whose monodromy is the crosscap transposition is not stably equivalent to the standard genus one open book with connected binding, whose monodromy  is the identity map.

Using algebraic methods, Stallings \cite{s2} proved  that the result of a Murasugi sum (a.k.a. plumbing) of  the {\em oriented} pages of two open books is the oriented page of another open book whose total space is the connected sum of total spaces of those open books. Later, Gabai \cite{g} gave a geometric proof of Stallings' theorem. Here, we formulate a version of Stallings' theorem without imposing the orientability of the  pages of the open books so that Gabai's proof extends to this version. 

We also provide a method to find a presentation of the fundamental group of the total space of any given nonorientable open book (in dimension three) based on the fundamental group of its page and its monodromy, similar to the one given in \cite{eo} for the orientable case. 


\section{$Y$-homeomorphism of Lickorish}\label{sec: yhomeo}
Throughout the paper, we denote by $K$ the {\em Klein bottle with one hole}, i.e., a nonorientable genus two surface with one boundary component. A model for $K$ is given in Figure~\ref{fig: kleinbottle}, where the edges of the octagon are identified as indicated by the arrows and the union of the red arcs is the boundary  $\partial K$ after the identifications.

\begin{figure}[h]
\centering
\includegraphics{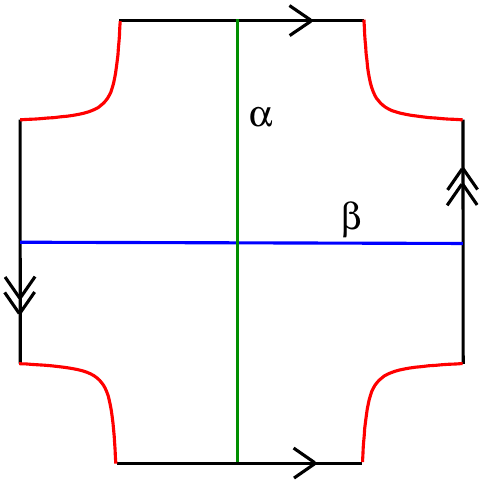}
  \caption{Klein bottle with one hole}
                        \label{fig: kleinbottle}
\end{figure}

Let $r$ denote the self-homeomorphism of $K$ induced by the reflection of this octagon through the circle $\a$. Since $r$ does not fix $\partial K$ pointwise, it is not an element of the mapping class group $Map(K)$ of $K$. However, the homeomorphism $y: K \to K $ obtained by post-composing $r$ with a half twist around $\partial K$ fixes $\partial K$ pointwise and hence it is an element 
of $Map(K)$, which is a $Y$-homeomorphism of Lickorish \cite{l}. The two possible half twists around $\partial K$ induce two $Y$-homeomorphisms of $K$ which are inverses of each other.   By definition,  $r^2$ is the identity, but this is no longer true for $y$. It can be easily verified that $ y^2$ is isotopic to a Dehn twist around  $\partial K$.

In fact, Lickorish defined a $Y$-homeomorphism on any nonorientable surface $N$ of genus at least two as follows.  Suppose that $\a$ is a two-sided and $\b$ is a one sided simple closed curve on $N$  intersecting transversely at one point.  Then a regular neighborhood of $\a \cup \b$ is homeomorphic to $K$. A $Y$-homeomorphism of $N$ is obtained simply by extending the homeomorphism $y$ to the surface $N$ by the identity map. A  $Y$-homeomorphism is also called a {\em crosscap slide} in the literature \cite{ko}, especially when the nonorientable surface at hand is described by crosscaps.  Note that   a  $Y$-homeomorphism on $N$ cannot be expressed as a product of Dehn twists along two-sided curves \cite{l}. 

\section{Klassen's genus two open book} \label{sec: cotlef}

In this section, we describe Klassen's  genus two open book for $P^2 \times S^1$ (cf. \cite{k}) and prove that its monodromy is the $Y$-homeomorphism of Lickorish. We view the real projective plane $P^2$ as the unit disk $D^2$ in the complex plane  $\mathbb{C}$ whose boundary is identified with itself by the antipodal map.  Note that the homeomorphism $\varphi' : D^2 \to  D^2 $ given by $\varphi' (z) = -z$ descends to a homeomorphism $\varphi : P^2 \to  P^2 $, which is isotopic to the identity. As a matter of fact, any self-homeomorphism of $P^2$ is isotopic to the identity \cite{e}. Let $I$ denote the unit interval $[0,1]$.  It follows that the quotient $P^2 \times I / \big((x, 1) \sim (\varphi(x), 0)\big)$, is homeomorphic to $P^2 \times S^1$.

Next, we consider the map $p: \big( D^2 - \{-\frac{1}{2}, \frac{1}{2}\}\big) \times I \to S^1$ which is given by $$p(z,t) = \dfrac{z^2 -\frac{1}{4}}{|z^2 -\frac{1}{4}|} e^{2\pi i t}.$$ Note that $\big(\{-\frac{1}{2}, \frac{1}{2}\} \times I\big)/\sim$ is a knot in $P^2 \times S^1$, which we denote by $B$.  The map $p$ induces a fibration of the complement of $B$  in $P^2 \times S^1$ over the unit circle in the complex plane.  This construction yields an open book for $P^2 \times S^1$ whose binding $B$ winds twice around the $S^1$ factor.

Now we consider the page $F_0 \subset P^2 \times S^1$ which is defined as the union of the binding $B$ and the fiber $p^{-1}(e^{2\pi i 0})$ under the identifications described above. We claim that $F_0$  is a Klein bottle with one hole, whose boundary is $B$. By definition, the interior of $F_0$  is the solution set of the equation $p(z,t)=1$, or equivalently, the equation 
\begin{equation}\label{eq: map}
(z^2 -\frac{1}{4}) e^{2\pi i t} =|z^2 -\frac{1}{4}|. \end{equation}  To visualize this solution set, which is a surface,  we compute its cross sections for each $ t \in I$ in the cylinder  $\big( D^2 - \{-\frac{1}{2}, \frac{1}{2}\}\big) \times I$, keeping in mind that the points on $\partial D^2$ are identified by the antipodal map.  Since the right-hand side of  Equation~\ref{eq: map} is real and positive, by plugging in $z=x+iy$, we conclude that \begin{equation} \label{eq: zero} Im [(x^2 +2xyi -y^2 -\frac{1}{4})\big(\cos (2 \pi t) + i \sin (2\pi t)\big)] =0 \end{equation} and 
\begin{equation} \label{eq: pos} Re [(x^2 +2xyi -y^2 -\frac{1}{4})\big(\cos (2 \pi t) + i \sin (2\pi t)\big)] > 0 \end{equation} which, in turn, imply that   \begin{equation}\label{eq: real}  (x^2 -y^2 -\frac{1}{4})\sin (2\pi t) + 2xy \cos (2 \pi t) =0.\end{equation} and  \begin{equation} \label{eq: part} xy <0 \; \mbox{if} \; t \in  (0,  \frac{1}{2})  \; \mbox{and} \; xy >0 \; \mbox{if} \; t \in  (\frac{1}{2}, 1). \end{equation}

\begin{figure}[h]
\centering
\includegraphics[scale=0.3]{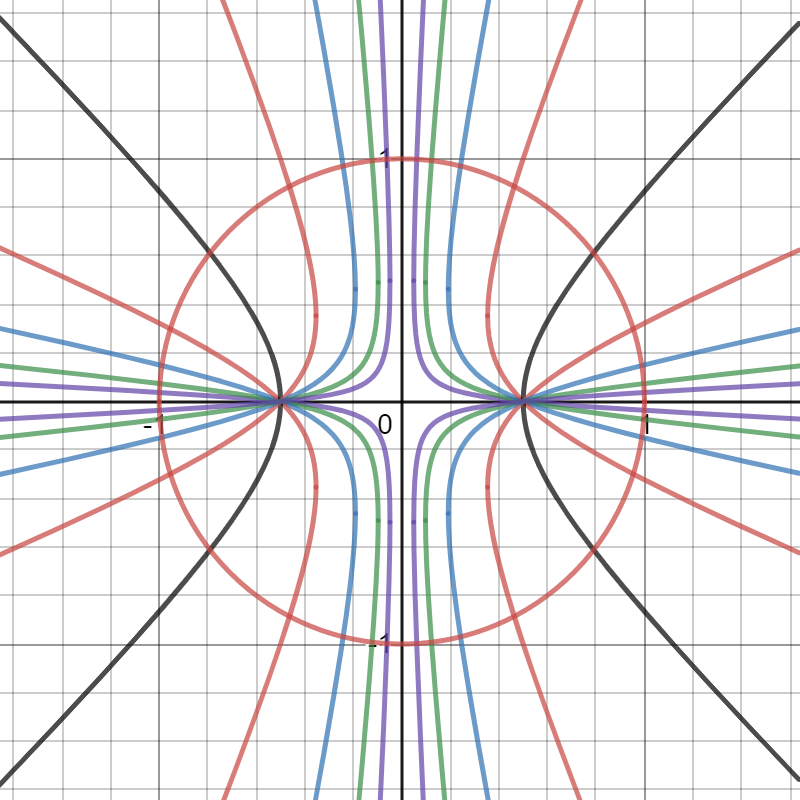}
  \caption{The hyperbolas $ (x^2 -y^2 -\frac{1}{4})\sin (2\pi t) + 2xy \cos (2 \pi t) =0$ for $t \in (0,1)$ and $t \neq \frac{1}{2}$. }
                        \label{fig: hyperbolas}
\end{figure}

We observe that, for each $  t \in (0,  \frac{1}{2}) \cup (\frac{1}{2}, 1)$, the solution of  Equation~\ref{eq: real} is a hyperbola whose axis is rotated by the angle $(\frac{1}{4} - t)\pi$ as depicted in Figure~\ref{fig: hyperbolas}. Note that each of these hyperbolas pass through the points $\pm \frac{1}{2}$ in the plane.   Therefore,  by Inequality ~\ref{eq: part},  the cross section at any $t \in  (0,  \frac{1}{2})$ is the part of the corresponding hyperbola that belongs to the region $xy <0$  in the plane, which becomes the interior of an arc connecting two points in $B$ via the identification of $\partial D^2$ under the antipodal map. Similarly,  the cross section at  $t \in  (\frac{1}{2}, 1)$  is the part of the corresponding hyperbola that belongs to the region $xy >0$ in the plane, which becomes the interior of an arc connecting two points in $B$ via the identification of $\partial D^2$ under the antipodal map.

When $t=0$, Equation~\ref{eq: zero} and Inequality~\ref{eq: pos} imply  that the cross section is the union $(-1, -\frac{1}{2}) \cup (\frac{1}{2}, 1)$ on the real line in $D^2 \subset \mathbb{C}$ as depicted on the left in Figure~\ref{fig: twolevel}, which becomes the interior of an arc connecting two points in $B$ via the identification of $\partial D^2$ under the antipodal map.  The cross section  for $t=1$ is the same as the cross section for $t=0$ and these arcs are identified in $P^2 \times S^1$ under the map $\varphi$ via reversing the orientation. 

\begin{figure}[h]
\centering
\includegraphics[scale=.7]{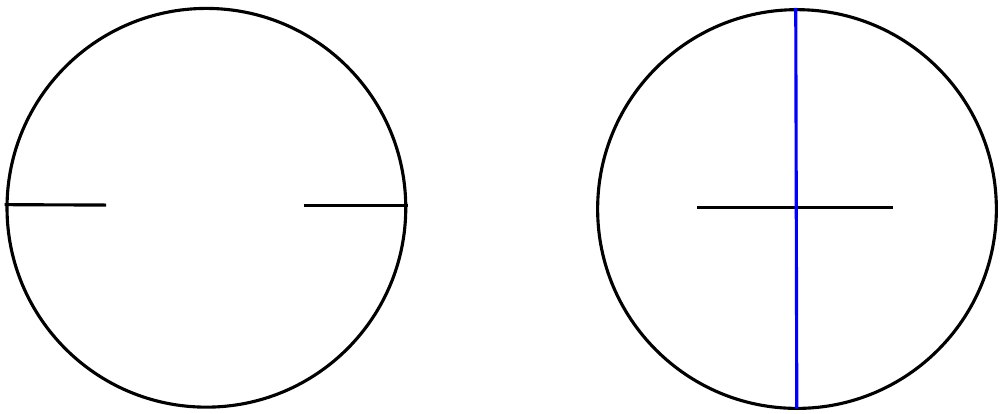}
  \caption{On the left: Cross section at $t=0$, and on the right: Cross section at $t =\frac{1}{2}$}
       \label{fig: twolevel}
\end{figure}

When $t=\frac{1}{2}$, Equation~\ref{eq: zero} and Inequality~\ref{eq: pos} imply that the cross section  is the union of the arc $(-i,i)$ on the imaginary axis and the arc $(-\frac{1}{2}, \frac{1}{2})$ on the real axis  as depicted on the right in Figure~\ref{fig: twolevel}, which becomes the union of a circle and the interior of an arc connecting two points in $B$ via the identification of $\partial D^2$ under the antipodal map.

\begin{figure}[h]
\centering
\includegraphics{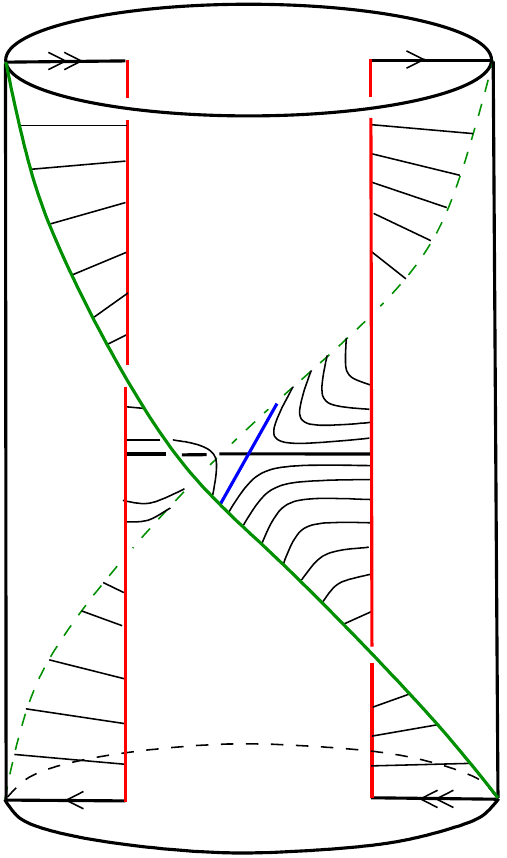}
  \caption{The page $F_0$  is depicted in $P^2 \times S^1$. Here each horizontal disk is a copy of the unit disk in $\mathbb{C}$ that  represents a $P^2$. The top ($t =1$) and the bottom ($t=0$) $P^2$ are identified by the homeomorphism $\varphi$ to obtain  $P^2 \times S^1$. There is a saddle of the surface $F_0$ at the level $t=\frac{1}{2}$. The union of the red arcs is the binding $B =\partial F_0 \subset P^2 \times S^1$. }
       \label{fig: cylinder}
\end{figure}

To summarize, for each $t \in [0, \frac{1}{2}) \cup (\frac{1}{2}, 1]$ the cross section is an arc connecting two points on the binding $B$ and for $t=\frac{1}{2}$ the cross section is the union of a circle and an arc  connecting two points in $B$. The cross section at the critical level $t=\frac{1}{2}$ can be viewed as the limit of the hyperbolas depicted in Figure~\ref{fig: hyperbolas} as $ t \to \frac{1}{2}$. The cross section for $t=0$ (which is the same as $t=1$) can also be seen in Figure~\ref{fig: hyperbolas} as the limit of the hyperbolas as $t \to 0$ (or equivalently as $t \to 1$). We claim that the page $F_0$, which is the union of the binding $B$ and all the cross sections for $t \in [0,1]$ under the identification $\varphi$ is homeomorphic to $K$, the Klein bottle with one hole, whose boundary is $B$. To see that $F_0$ depicted in Figure~\ref{fig: cylinder} is homeomorphic to $K$, we simply redraw it as an octagon on the plane with appropriate identification of its edges as shown in Figure~\ref{fig: identify}.

\begin{figure}[h]
\centering
\includegraphics[scale=1]{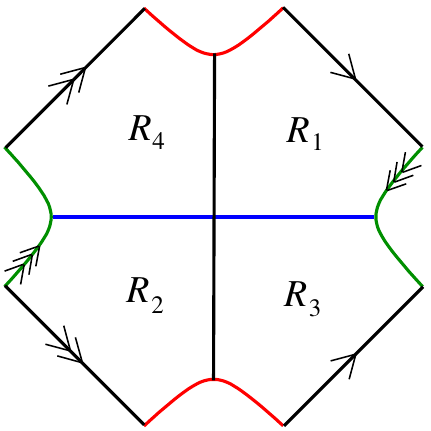}
  \caption{After identifying the edges as indicated,  the result is homeomorphic to $K$. The union of the red arcs  gives $\partial K$. }
       \label{fig: identify}
\end{figure}

 Next, we observe that, for any $s \in (0,1)$,  the page $F_s = B \cup  p^{-1} (e^{2\pi i s})$ of this open book is obtained by a vertical translation of the page $F_0$ depicted in Figure~\ref{fig: cylinder}. This follows from the fact that the interior of the page $F_s$ is the solution set of the equation $$P(z,t) = \frac{z^2 - \frac{1}{4}}{|z^2- \frac{1}{4}|} e^{2\pi i t} = e^{2 \pi i s} $$ which is equivalent to $$  \frac{z^2 - \frac{1}{4}}{|z^2- \frac{1}{4}|} e^{2\pi i (t-s)} = 1.$$ Therefore, the cross section of the page $F_s$ at the level $t$ is  exactly the cross section of $F_0$ 
at the level $t-s$.

\begin{theorem}\label{thm: main} The monodromy of  Klassen's  open book for $P^2 \times S^1$, whose page is a Klein bottle with one hole, is given by the $Y$-homeomorphism of Lickorish.  \end{theorem} 

\begin{proof} Using the description of $P^2 \times S^1$ as in Figure~\ref{fig: cylinder}, one can see that the vertical vector field $\frac{\partial}{\partial t}$ is transverse to the (interior) of each page of the open book  and tangent to the binding $B=\partial F_0$. Now we consider the homeomorphism $h$ of the page $F_0$  induced by the time $1$-map of the flow of this vector field. Note that $h$ interchanges the two red arcs in Figure~\ref{fig: cylinder}. In other words, although it does not fix the binding $B$ pointwise, it takes $B$ to itself by half twist (a rotation of $\pi$ degrees).

The solid green arc appearing at the front face of the cylinder in Figure~\ref{fig: cylinder} is mapped by $h$  to the dashed green arc at the back of the cylinder  and vice versa. But the solid green arc and the dashed green arc (which are in fact both circles after identifications) are identified with each other on the page $F_0$. This means that the homeomorphism $h$  takes the green circle to itself. The flow of the vertical vector field preserves the blue circle shown at the critical level $t = \frac{1}{2}$ in Figure~\ref{fig: cylinder}, except that its orientation gets reversed (via a reflection through its midpoint) while identifying the top and the bottom $P^2$ by the map $\varphi$.

\begin{figure}[h]
\centering
\includegraphics{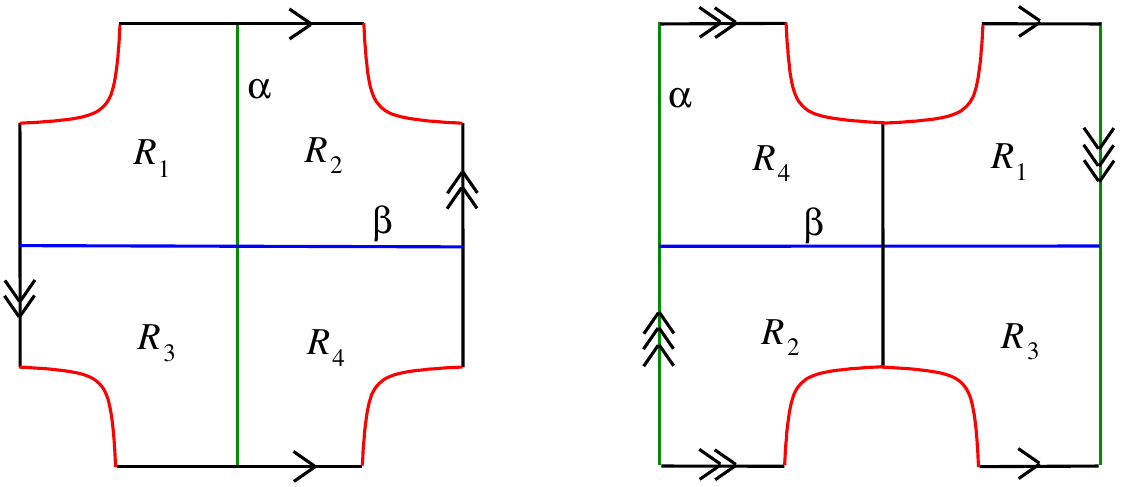}
  \caption{Two models of $K$. The one on the right (which is  a copy of Figure~\ref{fig: identify}) is obtained by cutting the one on the left (which is a copy of   Figure~\ref{fig: kleinbottle}) along $\a$ and gluing back the resulting pieces as indicated.  }
       \label{fig: kb}
\end{figure}

We conclude that, in the model of $K$ depicted in Figure~\ref{fig: identify}, the homeomorphism $h$ is a reflection of $K$ that interchanges the region $R_1$ with the region $R_2$, and the region $R_3$ with the region $R_4$.  This reflection in the model of $K$ in Figure~\ref{fig: identify} is equivalent to the reflection through $\a$ in Figure~\ref{fig: kleinbottle}, as illustrated in Figure~\ref{fig: kb}.


Hence the homeomorphism $h$ on $F_0 = K$ induced by the vertical vector field is equivalent to the reflection homeomorphism $r$ defined in Section~\ref{sec: yhomeo}. Therefore, by perturbing the vector field near the binding $B$  so that it is meridional (and hence its flow fixes $B$ pointwise),  we conclude that the monodromy as an element of $Map(K)$, is the $Y$-homeomorphism $y: K \to K$ of Lickorish.  \end{proof}


\section{Another proof of Theorem~\ref{thm: main}}

In this section, we provide an alternative proof of Theorem~\ref{thm: main}. We denote by $OB(N,\phi)$ the closed $3$-manifold which is the {\em total space} of an abstract open book whose page is a surface  $N$, and whose monodromy is $\phi \in Map (N)$.


\begin{theorem}\label{thm: openb} If $y \in Map (K)$ is the $Y$-homeomorphism of Lickorish, then $OB(K, y)= P^2 \times S^1$. \end{theorem} 

\begin{proof} If $p$ is a point  in $\partial K$,  then $\pi_1 (K, p)$ is freely generated  by the oriented loops $a$ and $b$ depicted in Figure~\ref{fig: fund}. First we compute the fundamental group of the mapping torus $K_y= K \times [0,1] / \sim$ based at the point $p$,  where $(x, 1) \sim (y(x), 0)$. Let $\t$ denote the loop $[0,1] / \sim$ based at $p$ and let $y_*$ denote the action of $y$ on  $\pi_1 (K_y, p)$. Then $$\pi_1(K_ y, p) = \langle a,b, \t \;|\; \t a \t ^{-1}= y_*a, \t b \t ^{-1}=y_*b \rangle.$$  

\begin{figure}[h]
\centering
\includegraphics{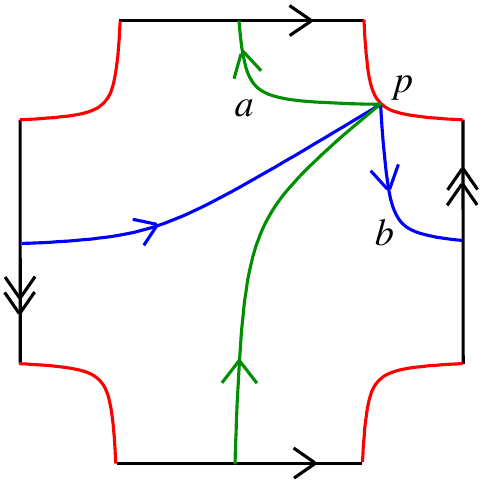}
 \caption{ $\pi_1 (K, p)$ is freely generated by the oriented loops $a$ and $b$.}
       \label{fig: fund}
\end{figure}

\begin{figure}[h]
\centering
\includegraphics{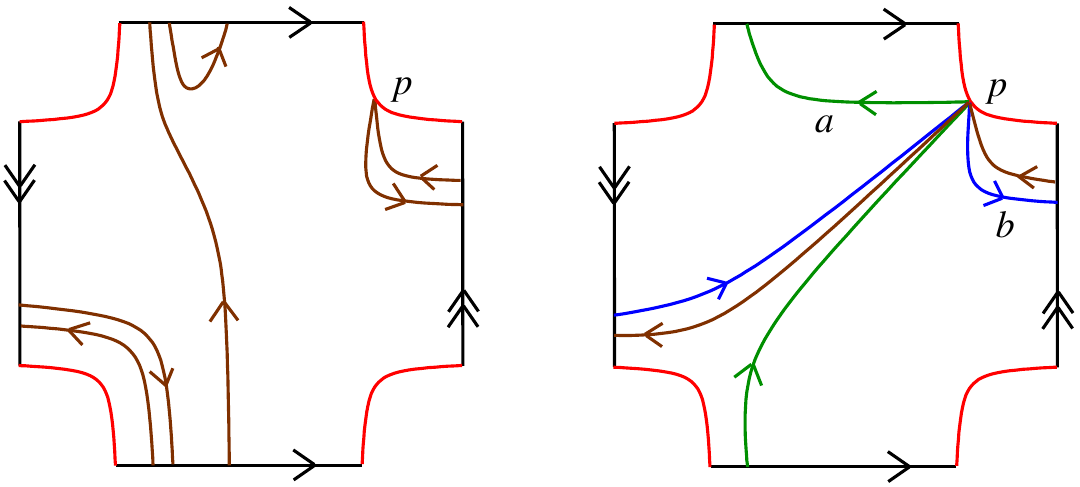}
  \caption{The loop $y_*a$ (left) is homotopic to $bab^{-1}$ (right).  }
       \label{fig: ima}
\end{figure}

\begin{figure}[h]
\centering
\includegraphics{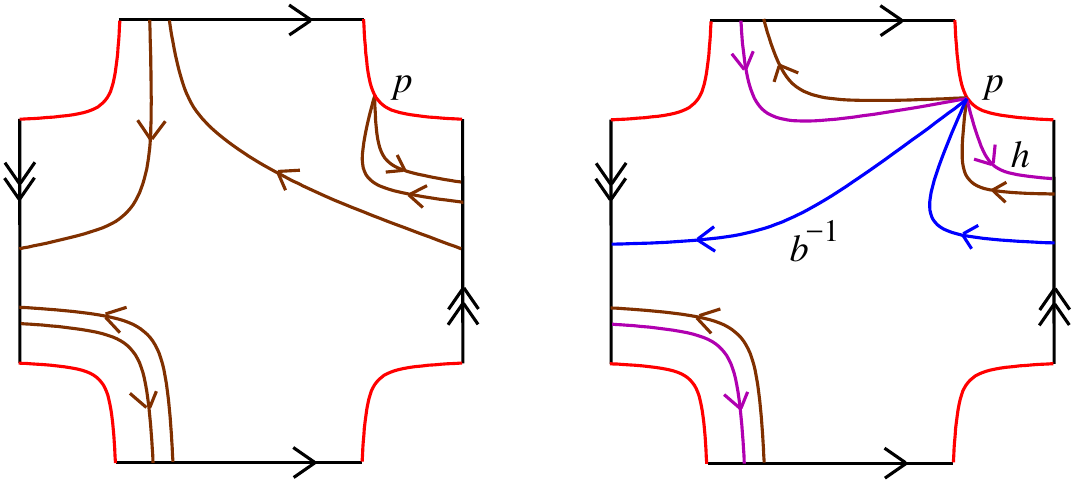}
  \caption{The loop $y_*b$ (left) is homotopic to $hb^{-1}h^{-1}$ (right).  }
       \label{fig: imb}
\end{figure}

It can be verified that $y_*a$ is homotopic to  $bab^{-1}$ (see Figure~\ref{fig: ima})  and $y_*b$ is homotopic to  $ hb^{-1}h^{-1}$  (see Figure~\ref{fig: imb}), for some loop $h$ based at $p$.  To obtain $OB(K, y)$ from the mapping torus $K_y$, we just glue a solid torus (a neighborhood of the binding) where the loop $\t$ is null-homotopic.  Therefore, the fundamental group of $OB(K, y)$ based at $p$ has the following presentation $$\langle a, b \; | \; a=bab^{-1} , \; b= hb^{-1}h^{-1} \rangle.$$  This group is abelian because of the first relation and the second relation gives $b^2=1$.  Hence $\pi_1 (OB(K, y), p)$  is isomorphic to $  \mathbb{Z} \times \mathbb{Z}_2 $. It is well-known that a closed $3$-manifold whose fundamental group is isomorphic to $  \mathbb{Z} \times \mathbb{Z}_2 $ is homeomorphic to $P^2 \times S^1$ (see, for example, Table 1.2 in \cite{afw}).  Here, the {\em Poincar\'{e} conjecture} is needed to rule out the case of connected sum with a homotopy sphere.  There is, however,  an alternative proof avoiding the use of the Poincar\'{e} conjecture as follows:  The $3$-manifold $OB(K, y)$ has a (nonorientable) Heegaard splitting of genus two obtained from the genus two open book and Ochiai \cite{oc} proved that any closed nonorientable $3$-manifold with fundamental group $  \mathbb{Z} \times \mathbb{Z}_2 $, which admits a Heegaard splitting of genus two,  is  homeomorphic to $P^2 \times S^1$.  \end{proof}

\section{Fundamental group of a nonorientable open book}

The total space of an open book is orientable if and only if the pages of the open book are orientable. Therefore, we will refer to an open book with  nonorientable  pages as a  {\em nonorientable open book}, in short. The genus of an open book is defined as the genus of its page as a (not necessarily orientable) surface with boundary.


In the following, we briefly review a method to calculate the fundamental group of the total space of any abstract nonorientable open book in dimension three. The discussion below is similar to the orientable case described in \cite[Section 2.1]{eo}. Suppose that $N$ is a nonorientable surface of genus $k$ with $r$ boundary components,  and let $\phi \in Map (N)$.  Then the total space 
 of the abstract open book with page $N$ and monodromy $\phi$,  is denoted by $OB(N,\phi)$, which is a closed nonorientable $3$-manifold. 

For all $1 \leq j \leq  r$, let $p_j$ be a point on the $j$-th boundary component of $\partial N$. Let $a_1,\ldots, a_k, c_1, \ldots, c_r$ be
the standard generators of $\pi_1(N, p_1)$, where
$c_i$'s correspond to loops around the boundary components.   Let $\phi_*$ denote the action induced by  $\phi$ on $\pi_1(N, p_1)$ and let $\t_j$ be the loop $[0,1] / \sim$ based at $p_j$ in the mapping torus $N_\phi =  N \times [0,1] / \sim$, where $(x, 1) \sim (\phi(x), 0)$. Then the fundamental group of $N_\phi$ based at $p_1$ has the following presentation:  $$ \langle a_1,\ldots, a_k, c_1, \ldots, c_r,  \t_1 \;|\; \prod_{i=1}^k  a_i^2 \prod_{j=1}^r c_j =1,\; \t_1 a_i \t_1 ^{-1}= \phi_*a_i, \; \t_1 c_j \t_1 ^{-1}=\phi_*c_j \rangle.$$  
For each $1 \leq j \leq  r$, let $\d_j \subset N$ be an arc connecting the base point $p_1$ to $p_j$. It follows that $\t_1 \d_j \t_j^{-1} \phi_* (\d_j^{-1})$ bounds a disk in $N_\phi$. To obtain $OB(N, \phi)$ from $N_\phi$, we glue in $r$ copies of the solid torus where $\t_j$ becomes null-homotopic for all $j =1, \ldots, r$. Consequently, we get a presentation of $\pi_1(OB(N, \phi), p_1)$  as follows:
$$ \langle a_1, \ldots, a_k, c_1, \ldots, c_r  \; \vert \; \prod_{i=1}^k  a_i^2 \prod_{j=1}^r c_j=1, \;a_i =\phi_*(a_i),\;  \; \d_j=\phi_*(\d_j) \rangle.$$
One can, of course, calculate the first homology group $H_1(M)$ by abelianizing $\pi_1(M)$. 
  

\section{Murasugi Sum of nonorientable open books} 

Suppose that $M_i = OB(\Sigma_i, \phi_i)$ where $\Sigma_i$ is an {\em oriented} surface and $\phi_i \in Map(\Sigma_i)$ for $i=1,2$. Stallings \cite{s2} proved, using the language of fibered knots instead of open books, that $M_1 \# M_2 = OB(\Sigma_1 * \Sigma_2, \phi_1 \circ \phi_2)$, where $\Sigma_1 * \Sigma_2$ denotes the Murasugi sum (a.k.a. plumbing) of the pages $\Sigma_1 $ and $\Sigma_2$.  In the following,  we state Stallings'  theorem without any assumption on the orientability of the pages but Gabai's geometric proof  \cite{g} of Stallings'  theorem holds true as long as the  Murasugi sum is performed along two-sided arcs on the pages.

\begin{proposition} \label{prop: gabai} Suppose that $M_i = OB(N_i, \phi_i)$ where $N_i$ is a surface which is not necessarily orientable and $\phi_i \in Map(N_i)$ for $i=1,2$. Then $M_1 \# M_2 = OB(N_1 * N_2, \phi_1 \circ \phi_2)$, provided that the Murasugi sum $N_1 * N_2$ is performed along two-sided arcs on $N_1$ and $N_2$. 
\end{proposition} 

A proof of Proposition~\ref{prop: gabai} can be obtained by adapting Gabai's geometric proof  in the orientable case described with a different point of view in Etnyre's notes \cite{et}. One can also consult \cite{op}, for  a more general approach.  
Based on Proposition~\ref{prop: gabai},  below we identify the total spaces of some open books with page $K$, but first we prove a simple result for an arbitrary nonorientable surface with boundary.

\begin{lemma} \label{lem: mob} Let $N_{g,k}$ denote a nonorientable surface of genus $g$ with $k \geq 1$ boundary components. Then $$OB(N_{g, k}, id)= \#_g S^2 \widetilde{\times} S^1 
\#_{k-1}  S^2 \times S^1 = \#_{g+k-1} S^2 \widetilde{\times} S^1 $$ where $S^2 \widetilde{\times} S^1$  denotes the nonorientable  
$S^2$-bundle over $S^1$. \end{lemma}
\begin{proof}  Let  $\A$ denote an annulus, and $\M$ denote a M\"{o}bius band. Then it is easy to see that $OB(\A, id)= S^2 \times S^1$ and $OB(\M, id)= S^2 \widetilde{\times} S^1$. The result follows from Proposition~\ref{prop: gabai}, since  $N_{g,k}$ can be obtained by a Murasugi sum of $g$ copies of $\M$ and $k-1$ copies of $\A$. \end{proof} 

Note that an alternative $4$-dimensional proof of Lemma~\ref{lem: mob} can be obtained as follows. The $4$-manifold $D^2\times N_{g,k}$ is obtained by attaching $g$ nonorientable and $k-1$ orientable $1$-handles to $D^4$ and hence it is diffeomorphic to   
$\natural_{g} D^3 \widetilde{\times}  S^1 \natural_{k-1}  D^3 \times S^1$, where  $D^3 \widetilde{\times} S^1$ denotes the nonorientable 
$D^3$-bundle over $S^1$. Therefore, we obtain a diffeomorphism between their boundaries $$  \#_g S^2 \widetilde{\times} S^1 
\#_{k-1}  S^2 \times S^1 = \partial  (\natural_{g} D^3 \widetilde{\times} S^1 \natural_{k-1}  D^3 \times   S^1) = \partial (D^2\times N_{g,k} ) = OB(N_{g,k}, id).$$

\begin{proposition}\label{prop: lens} For any $n \in \Z$, we have $$OB(K, t_\a^n) = L(|n|, 1) \# S^2 \widetilde{\times} S^1$$ where $t_\a$ denotes the Dehn twist along the curve $\alpha$ in Figure~\ref{fig: plumb}.  In particular,  $OB(K, id) = S^2 \times S^1 \# S^2 \widetilde{\times} S^1= \#_2 S^2 \widetilde{\times} S^1$, and $ OB(K, t^{\pm 1}_\a)= S^2 \widetilde{\times} S^1.$ \end{proposition}

\begin{figure}[h]
\centering
\includegraphics[scale=0.5]{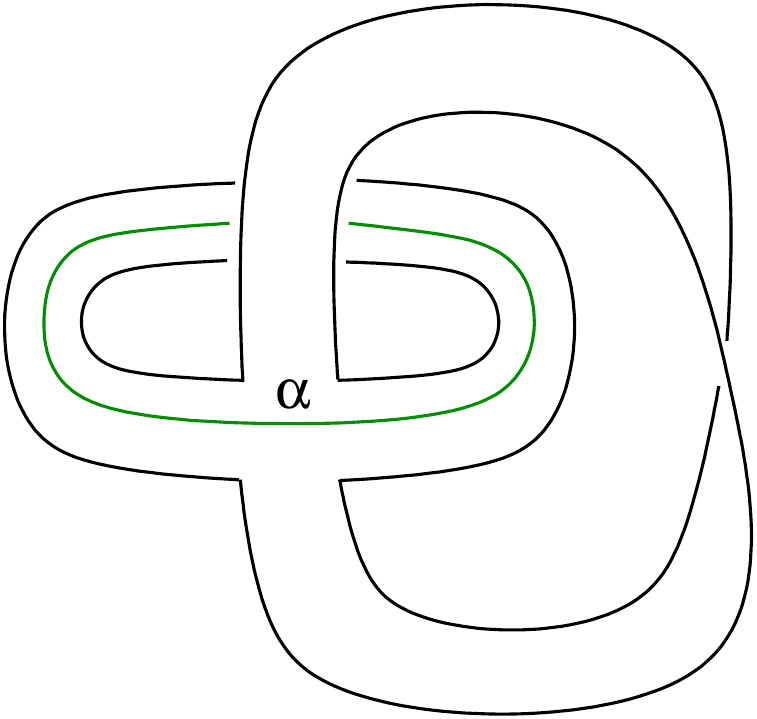}
  \caption{  }
       \label{fig: plumb}
\end{figure}

\begin{proof}  It is clear that $K$ can be obtained by the Murasugi sum of $\A$ and $\M$ as depicted in Figure~\ref{fig: plumb}.  Note that we have $OB(\A, t_\a^n) =L(n, n-1)$ for $n \geq 0$ and  $OB(\A, t_\a^n) =L(|n|, 1)$  for $n < 0$,  as {\em oriented} $3$-manifolds (cf. \cite[Lemma 5.1]{eo}). We also have $OB(\M, id) = S^2 \widetilde{\times} S^1$ by Lemma~\ref{lem: mob}.  Hence, by Proposition~\ref{prop: gabai},  we conclude that $OB(K, t_\a^n) = L(|n|, 1) \# S^2 \widetilde{\times} S^1$ for all $n \in \Z$, since there is an orientation-reversing homeomorphism between  the lens spaces $L(n,1)$ and $L(n, n-1)$ for $n >0$. \end{proof} 


\section{Infinitely many isomorphic nonorientable open books} 

Stukow \cite{st} showed that the mapping class group $Map(K)$ has a presentation with two generators:  the $Y$-homeomorphism $y$,  and the Dehn twist $t_{\a}$  about the two-sided curve $\a$ in Figure~\ref{fig: plumb}, with only one relation: $ t_{\a} y t_{\a} = y.$ It follows that any element of $Map(K)$ can be uniquely expressed as $t_{\a}^m y^n$ for some integers $m, n$.

\begin{definition} \cite{ps} The homeomorphism   $u=t_\a^{-1} y  \in Map(K)$ is called a {\em crosscap transposition.} \end{definition}

\begin{proposition} \label{prop: cross} For  $u \in Map(K)$, we have $OB(K, u)=  S^2 \widetilde{\times} S^1.$
\end{proposition} 
\begin{proof} The fundamental group of $OB(K, u)$ based at a point $p \in \partial K$ is isomorphic to $\Z$, which can be verified by a calculation similar to the one given in the proof  of  Theorem~\ref{thm: openb}.  Hence it  follows that $OB(K, u)=  S^2 \widetilde{\times} S^1$, as a consequence of the Poincar\'{e} conjecture. There is, however,  an alternative proof avoiding the use of the Poincar\'{e} conjecture as follows:  The $3$-manifold $OB(K, u)$ has a (nonorientable) Heegaard splitting of genus two obtained from the genus two open book and Ochiai \cite{oc} proved that any closed nonorientable $3$-manifold with fundamental group $  \mathbb{Z}$, which admits a Heegaard splitting of genus two,  is  homeomorphic to $ S^2 \widetilde{\times} S^1.$ 
\end{proof}

\begin{lemma} \label{lem: par} For any $m, n \in \Z$, the homeomorphism $t_\a^{2m} y^{2n+1} $ is conjugate to $y^{2n+1}$ and the homeomorphism  $t_\a^{2m+1} y^{2n+1} $ is conjugate to $uy^{2n}$ in $Map(K)$. \end{lemma} 
\begin{proof} For any $m \in \Z$, using the relation $t_\a  y t_\a = y$, we get $$t^{-m}_\a (t_\a^{2m} y) t^{m}_\a = t^{m}_\a  y t^{m}_\a = t^{m-1}_\a  y t^{m-1}_\a = \cdots =  t_\a  y t_\a = y .$$ Similarly, for any $m \in \Z$, we have $$ t_a^{-(m+1)} (t_\a^{2m+1}y) t_\a^{m+1} = t_\a^m y t_\a^{m+1} =\cdots= y t_\a=u. $$ Next, we use the facts that $y^2 = t_\partial$ (the boundary Dehn twist), and $t_\a$ commutes with $t_\partial$, to finish the proof. For any $m, n \in \Z$, we have $$ t^{-m}_\a (t_\a^{2m} y^{2n+1}) t^{m}_\a =  t^{-m}_\a (t_\a^{2m} y t_\partial^{n}) t^{m}_\a =  y t_\partial^{n} = y^{2n+1}.$$ Similarly, $$ t^{-(m+1)}_\a (t_\a^{2m+1} y^{2n+1}) t^{m+1}_\a =  t^{-(m+1)}_\a (t_\a^{2m} y t_\partial^{n}) t^{m+1}_\a =  u t_\partial^{n} = u y^{2n}.$$ \end{proof} 

\begin{corollary}\label{cor: conj} 
For any $m \in \Z$,  we have  
\begin{itemize} 
\item $OB(K, t_\a^{2m} y )=OB(K, y)= P^2 \times S^1 $, and 
\item $OB(K, t_\a^{2m+1} y )=OB(K, u)= S^2 \widetilde{\times} S^1$.  
\end{itemize}
\end{corollary}

\begin{proof} The total spaces of two open books with fixed page are homeomorphic, provided that the monodromies of these open books are conjugate in the mapping class group of the page. Therefore, Corollary~\ref{cor: conj} follows by combining Theorem~\ref{thm: openb}, Proposition~\ref{prop: cross}, and Lemma~\ref{lem: par}.
\end{proof} 

The next result immediately follows from Corollary~\ref{cor: conj}. 

\begin{corollary}  The product $P^2 \times S^1$ admits infinitely many isomorphic open books with page $K$ whose monodromies $\{t_\a^{2m} y \;| \; m \in \Z\}$ are mutually nonisotopic in $Map(K)$.  Similarly,  $S^2 \widetilde{\times} S^1$  admits infinitely many isomorphic open books with page $K$ whose monodromies $\{t_\a^{2m+1} y \;| \; m \in \Z\}$ are mutually nonisotopic in $Map(K)$. 
\end{corollary}

Let $\widehat{K}$ denote the (closed) Klein bottle and  $ \widehat{K} \widetilde{\times} S^1$ the twisted Klein bottle bundle over $S^1$ with monodromy $t_\a$. Recall that $t_{\partial}$ denotes the boundary Dehn twist in $Map(K)$.

\begin{proposition} For any $n \in \Z$, with $|n| \geq 1$, the $3$-manifold $OB(K, t^{\pm n}_\partial)$ is Seifert fibered over $\widehat{K}$. Moreover, we have  $OB(K, t^{\pm 1}_\partial)=\widehat{K} \widetilde{\times} S^1$. 
\end{proposition}

\begin{proof} The $\pm n$ surgery on a circle fiber of the bundle $\widehat{K} \times S^1$, yields an open book for the resulting $3$-manifold with page $K$ and monodromy $ t^{\pm n}_\partial$, similar to the orientable case discussed in \cite{oz}.    
For $|n| \geq 2$, the resulting nonorientable $3$-manifold admits a Seifert fibration over $\widehat{K}$ with one singular fiber, while for $|n|=1$, it is a circle bundle over $S^1$, which is homeomorphic to $ \widehat{K} \widetilde{\times} S^1$. 
\end{proof}

\section{Stable equivalence classes of nonorientable open books} 

In the following,  we assume that the $3$-manifolds are closed and connected but not necessarily orientable. 



A stabilization of an open book is defined as a finite sequence of Hopf plumbings and two 
open books  are called stably equivalent if they have isotopic stabilizations.  Every open book on a $3$-manifold induces a Heegaard splitting, where the Heegaard surface is the union of two distinct pages. Moreover,  if an open book is stabilized, then the associated Heegaard splitting is also stabilized. Reidemeister and Singer showed that any two Heegaard splittings of a $3$-manifold admit isotopic stabilizations.  Consequently, it is natural to ask whether any two open books for a given $3$-manifold are stably equivalent.

Using the celebrated Giroux correspondence between contact structures and open books, as an essential ingredient, Giroux and Goodman \cite{gg} gave a complete solution to this question in the orientable case:  ``Two open books for an oriented $3$-manifold admit isotopic stabilizations if and only if their associated oriented plane fields are homologous."


Although there is no contact structure on a {\em nonorientable} $3$-manifold, one can still consider an equivalence relation on the set of all open books for a nonorientable $3$-manifold, induced by stabilizations. 

\begin{corollary}  Each of $P^2 \times S^1$ and $S^2 \widetilde{\times} S^1$ admits a genus one open book and a genus two open book,  which are not stably equivalent. \end{corollary}

\begin{proof} According to \cite{gps},  $P^2 \times S^1$ admits a nonorientable genus one open book whose monodromy is a product of Dehn twists along two-sided curves. On the other hand, as we showed in Theorem~\ref{thm: main},  the monodromy of Klassen's genus two open book for $P^2 \times S^1$ is the $Y$-homeomorphism of Lickorish. In addition, the monodromy of any stabilization of Klassen's open book  will be the composition of the $Y$-homeomorphism with a product of  Dehn twists on two-sided curves. But since a $Y$-homeomorphism cannot be expressed as a product of Dehn twists, we conclude that Klassen's open book for  $P^2 \times S^1$ cannot be in the same stable equivalence class with any open book  whose monodromy is a product of Dehn twists, such as those described in \cite{gps}. Therefore, the aforementioned genus one open book and Klassen's genus two open book are not stably equivalent. 

As we discussed in Lemma~\ref{lem: mob}, there is a genus one open book for $S^2 \widetilde{\times} S^1$ with page the M\"{o}bius band and monodromy the identity map.  On the other hand,  by Proposition~\ref{prop: cross},  $S^2 \widetilde{\times} S^1$ also admits a genus two open book whose monodromy is the crosscap  transposition. Therefore, these two open books cannot be stably equivalent, by an argument similar to that given in the first paragraph of the proof. 
\end{proof}

\begin{corollary} Both Klassen's genus two open book for $P^2 \times S^1$ and 
the genus two open book for  $S^2 \widetilde{\times} S^1$ with monodromy the crosscap transposition cannot  be destabilized. \end{corollary}

\begin{proof} We observe that the genus two open books mentioned in the theorem cannot be obtained by stabilizing a genus one open book, since a $Y$-homeomorphism and hence a crosscap transposition only exist on a nonorientable surface of genus at least two.   \end{proof} 

\begin{remark}  We would like to point out that the ``only if" part of the Giroux-Goodman theorem holds true for nonorientable $3$-manifolds as well. Note that two oriented plane fields on an orientable $3$-manifold are homologous if and only if they are homotopic outside of a ball. Similar to the orientable case, there is a plane field associated to a nonorientable open book obtained by extending the tangent planes to the pages over the neighborhood of the binding, which is well-defined up to homotopy.  Moreover, Hopf plumbing yields an open book that coincides with the original one
in the complement of a ball and thus the associated plane field remains the same outside of a ball. We conclude that the associated plane fields of stably equivalent (orientable or nonorientable) open books are homotopic outside of a ball.  The proof of the converse direction of the Giroux-Goodman theorem, however,  relies heavily on contact geometry using in particular the isotopy classes of contact structures adapted to open books, rather than just the homotopy classes of the associated plane fields. So, it is not immediately clear how to modify the proof for nonorientable $3$-manifolds. \end{remark} 
 





\end{document}